\documentclass{amsart}
\title[Combinatorial Identities...]{Combinatorial Identities Involving Mertens Function Through Relatively Prime Subsets}
\usepackage{amssymb,amsmath,amsthm,epsfig,graphics,latexsym}

  \theoremstyle{plain}
  \newtheorem{definition}             {Definition}
  \newtheorem{lemma}      [definition]{Lemma}
  
  \newtheorem{theorem}    [definition]{Theorem}
  \newtheorem{corollary}  [definition]{Corollary}

  \theoremstyle{remark}

\begin{document}
  \author{Mohamed El Bachraoui}
  \address{Dept. Math. Sci,
 United Arab Emirates University, PO Box 17551, Al-Ain, UAE}
 \email{melbachraoui@uaeu.ac.ae}
  %
  \begin{abstract}
  In this note we give some identities which involve the Mertens function $M(n)$. Our proofs are combinatorial with relatively
  prime subsets as a main tool.
  \end{abstract}
  \date{\textit{\today}}
  \maketitle
  %
%
 \section{Introduction} \label{sec:introduction}
 Mertens function given by
 \[ M(n) =\sum_{d=1}^{n} \mu(d), \]
 where $\mu$ denotes the M\"obius mu function, is an important function in (analytic) number theory.
 Most of mathematical identities where $M(n)$ appears are either recursive formulas for $M(n)$ or formulas with an analytic
 flavor relating $M(n)$ to other functions. For a survey on identities involving $M(n)$
 we refer to \cite{Apostol, Dress} and for a survey on combinatorial identities
 involving other arithmetical functions we refer to \cite{Hall, Riordan} and their references.
 In this work we will give some other identities involving the function $M(n)$.
 Our proofs are combinatorial and based on relatively prime subsets of sets of positive integers.
 We list two of the identities which we intend to prove. Let
 $\lfloor x \rfloor$ denote the floor of $x$.

  \noindent
 1) If $n >3$, then
 \[ \sum_{d=1}^n \mu(d) 2^{\lfloor\frac{n}{d} \rfloor - \lfloor\frac{n-3}{d}\rfloor} =
  \begin{cases}
  3 + M(n),\ \text{if $n$ is even} \\
  4 + M(n),\ \text{if $n$ is odd}.
  \end{cases}
 \]

 \noindent
 2) If $1<m<n$, then
 \begin{multline*}
 \sum_{d=1}^{n+1} \mu(d) 2^{\lfloor \frac{n+1}{d}\rfloor - \lfloor \frac{n-1}{d}\rfloor +
  \lfloor \frac{m}{d}\rfloor - \lfloor \frac{m-1}{d}\rfloor} = \\
 \begin{cases}
  2 + M(n+1),\ \text{if $(m,n)>1$ and $(m,n+1)>1$}\\
  3 + M(n+1),\ \text{if $(m,n)=1$ and $(m,n+1)>1$ or $(m,n)>1$ and
           $(m,n+1)=1$} \\
  4 + M(n+1),\ \text{if $(m,n)=(m,n+1)=1$}.
  \end{cases}
  \end{multline*}
 \section{relatively prime subsets of $[l_1,m_1]\cup[l_2,m_2]$} \label{sec:relativelyprime}
 Throughout this section let $k$, $l$, $m$, $l_1$, $l_2$, $m_1$, and $m_2$ be positive integers such that $l\leq m$,
 $l_1\leq m_1$ and $l_2\leq m_2$, let
 $[l,m]=\{l,l+1,\ldots,m\}$, and let $A$ be a nonempty finite set of positive integers.
 The set $A$ is called \emph{relatively prime}
  if $\gcd(A) = 1$.
  \begin{definition}
  Let
  \[ f(A) = \# \{X\subseteq A:\ X\not=
  \emptyset\ \text{and\ } \gcd(X) = 1 \}
  \]
  and
  \[ f_k (A) = \# \{X\subseteq A:\ \# X=
  k \ \text{and\ } \gcd(X) = 1 \}.
  \]
 \end{definition}
 Nathanson in \cite{Nathanson} introduced among other functions $f(n)$ and $f_k(n)$ (in our terminology $f([1,n])$
 and $f_k([1,n])$ respectively) and found
 \begin{equation} \label{Nathanson:1}
 f([1,n])= \sum_{d=1}^n \mu(d)(2^{\lfloor \frac{n}{d} \rfloor}-1) \
 \text{and\ }
 f_k([1,n])= \sum_{d=1}^n \mu(d) \binom{\lfloor \frac{n}{d} \rfloor}{k}.
 \end{equation}
 Formulas for $f([m,n])$ and $f_k([m,n])$ are found in \cite{ElBachraoui1, Nathanson-Orosz}.
 %
 Recently Ayad and Kihel in \cite{Ayad-Kihel2} considered relatively prime subsets of sets which are in arithmetic progression
 and obtained formulas for $f([l,m])$ and $f_k ([l,m])$ as consequences since the integer interval $[l,m]$ is in arithmetic progression.
 However the authors' argument seems not to extend to unions of integer intervals.
 In this section we will give formulas for $f([l_1,m_1]\cup [l_2,m_2])$ and for $f_k([l_1,m_1]\cup [l_2,m_2])$.
 For the sake of completeness we include the following result which is a natural
 extension of \cite[Theorem 2 (a)]{ElBachraoui1} on M\"obius inversion for arithmetical functions of several variables.
 For simplicity of notation we let
 \[
 (\overline{m}_a,\overline{n}_b)= (m_1,m_2,\ldots,m_a,n_1,n_2,\ldots,n_{b})
 \]
 and
 \[
 \left( \frac{\overline{m}_a}{d}, \left\lfloor\frac{\overline{n}_b}{d}\right\rfloor \right) =
 \left( \frac{m_1}{d},\frac{m_2}{d}, \ldots, \frac{m_a}{d},
 \left\lfloor\frac{n_1}{d}\right\rfloor,\left\lfloor\frac{n_2}{d}\right\rfloor,\ldots,\left\lfloor\frac{n_{b}}{d}\right\rfloor\right).
 \]
 \begin{theorem} \label{thm:inversion}
 If $F$ and $G$ are arithmetical of $a+b$ variables,
 then
 \[
 G(\overline{m}_a,\overline{n}_b)=\sum_{d|(m_1,m_2,\ldots,m_a)}F
 \left(\frac{\overline{m}_a}{d},
\left\lfloor\frac{\overline{n}_b}{d}\right\rfloor \right)
\]
if and only if
\[
 F(\overline{m}_a,\overline{n}_b)=\sum_{d|(m_1,m_2,\ldots,m_a)}\mu(d)G\left(\frac{\overline{m}_a}{d},
 \left\lfloor\frac{\overline{n}_b}{d}\right\rfloor \right).
 \]
 \end{theorem}
 %
 We need the following three lemmas the proofs of which can be obtained using the same sort of idea and therefore
 we prove only the first one.
 \begin{lemma}\label{lem1}
 Let
 \[
 g(m_1,l_2,m_2)= \# \{X \subseteq [1,m_1]\cup [l_2,m_2]:\ l_2\in X\
 \text{and\ } \gcd(X)=1 \},
 \]
 \[
 g_k(m_1,l_2,m_2)=\# \{X \subseteq [1,m_1]\cup [l_2,m_2]:\ l_2\in X,\ |X| = k,\
 \text{and\ } \gcd(X)=1 \}.
 \]
 Then
 \[ (a)\quad
   g(m_1,l_2,m_2) = \sum_{d|l_2}\mu(d)
  2^{\lfloor m_1/d \rfloor + \lfloor m_2/d \rfloor- l_2/d}, \]
\[
 (b)\quad  g_k (m_1,l_2,m_2) = \sum_{d|l_2}\mu(d)
 \binom{\lfloor m_1/d \rfloor + \lfloor m_2/d \rfloor - l_2/d }{k-1}.
 \]
 \end{lemma}
\begin{proof}
(a) Let $\mathcal{P}(m_1,l_2,m_2)$ denote the set of subsets of
 $[1,m_1]\cup[l_2,m_2]$
 containing $l_2$ and let $\mathcal{P}(m_1,l_2,m_2,d)$ be the set of subsets $X$ of
 $[1,m_1]\cup[l_2,m_2]$ such that
 $l_2\in X$ and $\gcd(X) = d$. It is clear that the set $\mathcal{P}(m_1,l_2,m_2)$
 of cardinality $2^{m_1+m_2-l_2}$  can be
 partitioned using the equivalence relation of having the same $\gcd$ (dividing $l_2$).
 Moreover, the mapping
 $A \mapsto \frac{1}{d} X$
 is a one-to-one correspondence between
 $\mathcal{P}(m_1,l_2,m_2,d)$ and the
 set of subsets $Y$ of
 $[1, \lfloor m_1/d \rfloor ]\cup [l_2/d,\lfloor m_2/d \rfloor]$
 such that $l_2/d \in Y$ and $\gcd(Y)= 1$. Then
 \[
 2^{m_1+m_2-l_2} = \sum_{d|l_2} \# \mathcal{P}(m_1,l_2,m_2,d)=
 \sum_{d|l_2} g (\lfloor m_1/d \rfloor,l_2 /d,\lfloor m_2/d \rfloor),
 \]
 which by Theorem \ref{thm:inversion}  is equivalent to
 \[
 g(m_1,l_2,m_2) = \sum_{d|l_2}\mu(d)
 2^{\lfloor m_1/d \rfloor + \lfloor m_2/d \rfloor - l_2/d}.
 \]
 (b) Similarly
 \[ \binom{m_1 + m_2 -l_2}{k-1} = \sum_{d|l_2} g_k (\lfloor m_1/d \rfloor,l_2 /d,\lfloor m_2/d \rfloor), \]
 which by Theorem \ref{thm:inversion}  is equivalent to
 \[
 g_k(m_1,l_2,m_2) = \sum_{d|l_2}\mu(d)\binom{\lfloor m_1/d \rfloor + \lfloor m_2/d \rfloor - l_2/d}{k-1}.
 \]
 \end{proof}
  \begin{lemma} \label{lem2}
 Let
 \[
 h^{(1)}(l_1,m_1) = \# \{X\subseteq [l_1,m_1]:\ l_1 \in X\ \text{and\ } \gcd(X)=1 \},
 \]
 \[
 h^{(1)}_k(l_1,m_1) = \# \{X\subseteq [l_1,m_1]:\ l_1 \in X,\ \# X=k,\ \text{and\ } \gcd(X)=1 \}.
 \]
 Then
 \[ (a)\quad
 h^{(1)}(l_1,m_1) = \sum_{d|l_1} \mu(d) 2^{\lfloor m_1/d \rfloor - l_1/d} ,
 \]
 \[ (b)\quad
 h^{(1)}_k(l_1,m_1) = \sum_{d|l_1} \mu(d) \binom{\lfloor m_1/d \rfloor - l_1/d}{k}.
 \]
 \end{lemma}
 \begin{lemma} \label{lem3}
 Let
 \[
 h^{(2)}(l_1,m_1,l_2,m_2)= \# \{X \subseteq [l_1,m_1]\cup [l_2,m_2]:\ l_1, l_2\in X\
 \text{and\ } \gcd(X)=1 \},
 \]
 \[
 h^{(2)}_k(l_1,m_1,l_2,m_2)=\# \{X \subseteq [l_1,m_1]\cup [l_2,m_2]:\ l_1,l_2\in X,\ |X| = k,\
 \text{and\ } \gcd(X)=1 \}.
 \]
 Then
 \[
 (a)\quad h^{(2)}(l_1,m_1,l_2,m_2) = \sum_{d|(l_1,l_2)}\mu(d)
  2^{\lfloor \frac{m_1}{d} \rfloor + \lfloor \frac{m_2}{d} \rfloor- \frac{l_1+l_2}{d}}, \]
\[
 (b)\quad  h^{(2)}_k (m_1,l_2,m_2) = \sum_{d|(l_1,l_2)}\mu(d)
 \binom{\lfloor \frac{m_1}{d} \rfloor + \lfloor \frac{m_2}{d} \rfloor- \frac{l_1+l_2}{d}}{k-2}.
 \]
 \end{lemma}

\noindent
 We further need the following special case to prove the main theorem of this section.
 \begin{theorem}\label{thm:main1}
 We have
  \[
 \begin{split} (a)\quad
  f([1,m_1]\cup [l_2,m_2]) &= \sum_{d=1}^{m_2}\mu(d) (
   2^{\lfloor \frac{m_1}{d} \rfloor +\lfloor \frac{m_2}{d} \rfloor - \lfloor\frac{l_2 -1}{d} \rfloor}-1), \\
 (b)\quad
  f_k ([1,m_1]\cup [l_2,m_2]) &=
   \sum_{d=1}^{m_2} \mu(d)
   \ \binom{\lfloor \frac{m_1}{d} \rfloor +\lfloor \frac{m_2}{d} \rfloor - \lfloor\frac{l_2 -1}{d} \rfloor}{k}.
  \end{split}
  \]
 \end{theorem}
 \begin{proof}
 (a) Clearly
 \[
 \begin{split}
 f([1,m_1]\cup [l_2,m_2])
 & =
 f([1,m_2]) - \sum_{i=m_1 +1}^{l_2 -1} g(m_1,i,m_2) \\
 &=
 \sum_{d=1}^{m_2} \mu(d) (2^{\lfloor m_2/d \rfloor}-1) - \sum_{i=m_1 +1}^{l_2 -1}\sum_{d|i}
 \mu(d) 2^{\lfloor \frac{m_1}{d} \rfloor +\lfloor \frac{m_2}{d} \rfloor - \frac{i}{d}} \\
 &=
 \sum_{d=1}^{m_2} \mu(d) (2^{\lfloor m_2/d \rfloor}-1) - \sum_{d=1}^{l_2-1} \mu(d) 2^{\lfloor \frac{m_1}{d} \rfloor +\lfloor \frac{m_2}{d} \rfloor}
 \sum_{j=\lfloor \frac{m_1}{d} \rfloor +1}^{\lfloor \frac{l_2-1}{d} \rfloor} 2^{-j} \\
 &=
 \sum_{d=1}^{m_2} \mu(d) (2^{\lfloor m_2/d \rfloor}-1) -\sum_{d=1}^{l_2 -1}\mu(d) 2^{\lfloor \frac{m_1}{d} \rfloor +\lfloor \frac{m_2}{d} \rfloor}
 2^{-\lfloor \frac{m_1}{d} \rfloor}
 \left(1- 2^{-\lfloor \frac{l_2-1}{d}\rfloor +\lfloor \frac{m_1}{d}\rfloor} \right) \\
 &=
 \sum_{d=1}^{m_2} \mu(d) (2^{\lfloor m_2/d \rfloor}-1)-\sum_{d=1}^{m_2}\mu(d) 2^{\lfloor \frac{m_2}{d} \rfloor}
 \left(1- 2^{-\lfloor \frac{l_2-1}{d}\rfloor +\lfloor \frac{m_1}{d}\rfloor} \right)\\
 &= \sum_{d=1}^{m_2}\mu(d) (
   2^{\lfloor \frac{m_1}{d} \rfloor +\lfloor \frac{m_2}{d} \rfloor - \lfloor\frac{l_2 -1}{d} \rfloor}-1),
 \end{split}
 \]
 where the second identity follows by (\ref{Nathanson:1}) and Lemma \ref{lem1}.
 \\
 (b) We have
 \[
 \begin{split}
 f_k([1,m_1]\cup [l_2,m_2])
 &=
 f_k([1,m_2] - \sum_{i=m_1 +1}^{l_2 -1} g(m_1,i,m_2) \\
 &=
 \sum_{d=1}^{m_2} \mu(d) \binom{\lfloor m_2/d \rfloor}{k} - \sum_{i=m_1 +1}^{l_2 -1}\sum_{d|i}
 \mu(d) \binom{\lfloor \frac{m_1}{d} \rfloor +\lfloor \frac{m_2}{d} \rfloor - \frac{i}{d}}{k-1} \\
 &=
 \sum_{d=1}^{m_2} \mu(d) \binom{\lfloor m_2/d \rfloor}{k} - \sum_{d=1}^{m_2} \mu(d)
  \sum_{j=\lfloor \frac{m_1}{d} \rfloor +1}^{\lfloor \frac{l_2-1}{d} \rfloor}
   \binom{\lfloor \frac{m_1}{d} \rfloor +\lfloor \frac{m_2}{d} \rfloor - \frac{i}{d}}{k-1} \\
 &=
 \sum_{d=1}^{m_2} \mu(d) \binom{\lfloor m_2/d \rfloor}{k} - \sum_{d=1}^{m_2} \mu(d)
  \sum_{i=\lfloor \frac{m_1}{d} \rfloor +\lfloor \frac{m_2}{d} \rfloor - \frac{l_2 -1}{d}}^{\lfloor \frac{m_2}{d} \rfloor -1}
  \binom{i}{k-1} \\
 &=
 \sum_{d=1}^{m_2} \mu(d) \binom{\lfloor m_2/d \rfloor}{k} - \sum_{d=1}^{m_2} \mu(d)
  \left( \binom{\lfloor m_2/d \rfloor}{k} -
  \binom{\lfloor \frac{m_1}{d} \rfloor +\lfloor \frac{m_2}{d} \rfloor - \frac{l_2 -1}{d}}{k} \right) \\
 &=
 \sum_{d=1}^{m_2} \mu(d) \binom{\lfloor \frac{m_1}{d} \rfloor +\lfloor \frac{m_2}{d} \rfloor - \frac{l_2 -1}{d}}{k}.
 \end{split}
 \]
 This completes the proof.
 \end{proof}
 \noindent
 We are now ready to prove the main theorem of this section.
 \begin{theorem} \label{main2}
 We have
 \[ (a)\quad
 f([l_1,m_1]\cup[l_2,m_2]) = \sum_{d=1}^{m_2}\mu(d)(
  2^{\lfloor \frac{m_1}{d} \rfloor + \lfloor \frac{m_2}{d} \rfloor -
    \lfloor \frac{l_1-1}{d} \rfloor - \lfloor \frac{l_2-1}{d} \rfloor}-1),
 \]
 \[ (b)\quad
 f_k([l_1,m_1]\cup[l_2,m_2]) = \sum_{d=1}^{m_2}\mu(d)
  \binom{\lfloor \frac{m_1}{d} \rfloor + \lfloor \frac{m_2}{d} \rfloor -
    \lfloor \frac{l_1-1}{d} \rfloor - \lfloor \frac{l_2-1}{d} \rfloor}{k}.
 \]
 \end{theorem}
 \begin{proof}
 (a)\ Clearly
 \begin{multline*}
 f([l_1,m_1]\cup[l_2,m_2]) =
 f([1,m_1] \cup [l_2,m_2]) -
 \sum_{i=1}^{l_1 -1} \sum_{j=l_2}^{m_2} h^{(2)}(i,m_1,j,m_2) -
 \sum_{i=1}^{l_1-1} h^{(1)}(i,m_1) \\
  (3)\quad\quad\ =
 \sum_{d=1}^{m_2}\mu(d) (2^{\lfloor \frac{m_1}{d} \rfloor + \lfloor \frac{m_2}{d} \rfloor -
    \lfloor \frac{l_2-1}{d} \rfloor}-1) -
 \sum_{i=1}^{l_1 -1} \sum_{j=l_2}^{m_2} \sum_{d|(i,j)} \mu(d)
    2^{\lfloor \frac{m_1}{d} \rfloor + \lfloor \frac{m_2}{d} \rfloor - \frac{i+j}{d}} -
 \sum_{i=1}^{l_1-1} \sum_{d|i} \mu(d) 2^{\lfloor \frac{m_1}{d} \rfloor - \frac{i}{d}},
 \end{multline*}
 where the second identity follows by Theorem \ref{thm:main1}, Lemma \ref{lem2}, and
 Lemma \ref{lem3}.
 Rearranging the triple summation in identity (3), we get
 \[
 \begin{split}
 \sum_{i=1}^{l_1 -1} \sum_{j=l_2}^{m_2} \sum_{d|(i,j)} \mu(d)
    2^{\lfloor \frac{m_1}{d} \rfloor + \lfloor \frac{m_2}{d} \rfloor - \frac{i+j}{d}}
  &=
 \sum_{d=1}^{m_2} \mu(d) 2^{\lfloor \frac{m_1}{d} \rfloor + \lfloor \frac{m_2}{d} \rfloor}
  \sum_{i=1}^{\lfloor\frac{l_1-1}{d}\rfloor} 2^{-i}
   \sum_{j=\lfloor \frac{l_2-1}{d} \rfloor +1}^{\lfloor \frac{m_2}{d} \rfloor} 2^{-j} \\
 (4) \qquad\qquad\qquad \quad\qquad\qquad\qquad
 & =
  \sum_{d=1}^{m_2} \mu(d) 2^{\lfloor \frac{m_1}{d} \rfloor + \lfloor \frac{m_2}{d} \rfloor - \lfloor\frac{l_2-1}{d}\rfloor}
  (1- 2^{-\lfloor \frac{m_2}{d} \rfloor + \lfloor\frac{l_2 -1}{d} \rfloor})
  (1- 2^{-\lfloor \frac{l_1-1}{d} \rfloor}).
 \end{split}
 \]
 Similarly the last double summation in identity (3) gives
 \[
 \begin{split}
 \sum_{i=1}^{l_1-1} \mu(d) 2^{\lfloor \frac{m_1}{d} \rfloor}\sum_{i=1}^{\lfloor\frac{l_1-1}{d}\rfloor} 2^{- \frac{i}{d}}
 &=
  \sum_{d=1}^{l_1 -1} \mu(d) 2^{\lfloor \frac{m_1}{d} \rfloor}(1- 2^{-\lfloor \frac{l_1-1}{d} \rfloor}) \\
 (5)\qquad\qquad\qquad\qquad
 &=
 \sum_{d=1}^{m_2} \mu(d) 2^{\lfloor \frac{m_1}{d} \rfloor}(1- 2^{-\lfloor \frac{l_1-1}{d} \rfloor}).
 \end{split}
 \]
 Appealing to identities (3), (4), and (5) we find
 \[
 f([l_1,m_1]\cup[l_2,m_2]) = \sum_{d|n}\mu(d)
 2^{\lfloor \frac{m_1}{d} \rfloor + \lfloor \frac{m_2}{d} \rfloor -
    \lfloor \frac{l_1-1}{d} \rfloor - \lfloor \frac{l_2-1}{d} \rfloor}.
 \]
 This completes the proof of part (a). Part (b) follows by similar arguments.
 \end{proof}
 \begin{corollary}\label{corollary1} \emph{(Ayad-Kihel \cite{Ayad-Kihel2})}
 We have
 \[ (a)\quad f([l,m])= \sum_{d=1}^{m}(2^{\lfloor\frac{m}{d}\rfloor - \lfloor\frac{l-1}{d}\rfloor}-1), \]
 \[ (b)\quad f_k([l,m])= \sum_{d=1}^{m}\binom{\lfloor\frac{m}{d}\rfloor - \lfloor\frac{l-1}{d}\rfloor}{k}. \]
 \end{corollary}
 \begin{proof}
 Use Theorem \ref{main2} with $l_1=l$, $m_1=m-1$, and $l_2=m_2=m$.
 \end{proof}
%
 %
 \section{Combinatorial identities}
 \begin{theorem} \label{identities1}
 If $n>1$, then
 \[
  \sum_{d=1}^{n+1} \mu(d)( 2^{\lfloor \frac{n+1}{d} \rfloor - \lfloor \frac{n-1}{d} \rfloor} = 1 + M(n+1)
  \]
 \end{theorem}
 \begin{proof}
 Apply Corollary~\ref{corollary1} to the interval $[n,n+1]$ and use the obvious fact that $f([n,n+1])=1$.
 \end{proof}
  \begin{theorem} \label{identities2}
  In $n>3$, then
  \[ \sum_{d=1}^n 2^{\lfloor \frac{n}{d} \rfloor - \lfloor\frac{n-3}{d} \rfloor} =
  \begin{cases}
  3 + M(n),\ \text{if $n$ is even} \\
  4 + M(n),\ \text{if $n$ is odd}.
  \end{cases}
  \]
  \end{theorem}
 \begin{proof}
 Combine Corollary~\ref{corollary1} applied to the interval $[n-2,n]$ and the fact that
 \[
 f([n-2,n]) = \begin{cases}
 3,\ \text{if $n$ is even} \\
 4,\ \text{if $n$ is odd}.
 \end{cases}
 \]
 \end{proof}
 \begin{theorem} \label{identities3}
 (a) If $1<m < n$, then
 \begin{multline*}
 \sum_{d=1}^{n+1} \mu(d) 2^{\lfloor \frac{n+1}{d}\rfloor - \lfloor \frac{n-1}{d}\rfloor +
  \lfloor \frac{m}{d}\rfloor - \lfloor \frac{m-1}{d}\rfloor} = \\
 \begin{cases}
  2 + M(n+1),\ \text{if $(m,n)>1$ and $(m,n+1)>1$}\\
  3 + M(n+1),\ \text{if $(m,n)=1$ and $(m,n+1)>1$ or $(m,n)>1$ and
           $(m,n+1)=1$} \\
  4+ M(n+1),\ \text{if $(m,n)=(m,n+1)=1$}.
  \end{cases}
  \end{multline*}
 (b) If $1<n<m-1$, then
  \begin{multline*}
  \sum_{d=1}^{m} \mu(d) 2^{\lfloor \frac{n+1}{d}\rfloor - \lfloor \frac{n-1}{d}\rfloor +
  \lfloor \frac{m}{d}\rfloor - \lfloor \frac{m-1}{d}\rfloor} = \\
 \begin{cases}
  2 + M(m),\ \text{if $(m,n)>1$ and $(m,n+1)>1$}\\
  3 + M(m),\ \text{if $(m,n)=1$ and $(m,n+1)>1$ or $(m,n)>1$ and
           $(m,n+1)=1$} \\
  4 + M(m),\ \text{if $(m,n)=(m,n+1)=1$}.
  \end{cases}
  \end{multline*}
 \end{theorem}
 \begin{proof}
 (a) Clearly
 \[
 f([m,m]\cup [n,n+1])=
 \begin{cases}
  2,\ \text{if $(m,n)>1$ and $(m,n+1)>1$}\\
  3,\ \text{if $(m,n)=1$ and $(m,n+1)>1$ or $(m,n)>1$ and
           $(m,n+1)=1$} \\
  4,\ \text{if $(m,n)=(m,n+1)=1$}.
  \end{cases}
  \]
 Combine this identity with Theorem \ref{main2}(a) applied to $[m,m]\cup [n,n+1]$. \\
 (b) Similar to part (a) with an application of Theorem \ref{main2}(a) to $[n,n+1]\cup [m,m]$.
 \end{proof}
 %


\end{document}